\documentclass[11pt,a4paper,leqno]{amsart}
\usepackage{amsfonts}
\usepackage{mathrsfs}

\usepackage{amsthm}
\usepackage{xspace}
\usepackage{graphics}
\usepackage{amsfonts,amssymb}
\usepackage{amsthm}
\usepackage[all]{xy}
\usepackage{stmaryrd}
\usepackage{color}
\usepackage[francais,english]{babel}

\newcommand{\be}{\begin{otherlanguage}{english}}
\newcommand{\ee}{\end{otherlanguage}}

\theoremstyle{definition}

\newtheorem{rem}{Remark}
\theoremstyle{plain}
\newtheorem{lem}{Lemma}

\newtheorem{thm}[lem]{Theorem}

\newtheorem*{thm*}{Theorem}

\theoremstyle{remark}

\numberwithin{equation}{subsection}

\newcommand{\beq}{\begin{equation}}
\newcommand{\eeq}{\end{equation}}

\begin{document}
\title{A refinement of Franks' theorem}

\author{Hui Liu}
\address{Key Laboratory of Wu Wen-Tsun Mathematics, Chinese
Academy of Sciences, School of Mathematical Sciences, University of
Science and Technology of China, Hefei, Anhui 230026, P.R.China}
\email{huiliu@ustc.edu.cn}
\thanks{The first author is supported by NSFC (Nos. 11401555, 11371339), Anhui Provincial Natural Science Foundation
 (No. 1608085QA01); The second author is supported by China Postdoctoral Science Foundation (No.
 2013T60251), International Postdoctoral Exchange Fellowship Program (No. 20130045), NFSC (No. 11401320).}

\author{Jian Wang}
\address{Chern Institute of Mathematics and Key Laboratory of Pure Mathematics and Combinatorics
 of Ministry of Education, Nankai University, Tianjin 300071,
P.R.China} \email{wangjian@nankai.edu.cn} \curraddr{Max Planck
Institute for Mathematics in the Sciences, Inselstra{\ss}e 22,
D-04103 Leipzig, Germany} \email{jianwang@mis.mpg.de}

\date{Jan. 15, 2016}
\maketitle

\begin{abstract}In this paper, we give a refinement of Franks' theorem \cite{F2}, which
answers two questions raised by Kang \cite{Kang}.
\end{abstract}
\renewcommand\abstractname{R\'esum\'e}
\begin{abstract}Dans cet article, nous donnons
un raffinement du th\'eor\`eme de Franks \cite{F2}, il r\'epond aux
questions propos\'es par Kang  \cite{Kang}.
\end{abstract}
\section{introduction}
Let $\mathbb{A}=\mathbb{R}/\mathbb{Z}\times (0,1)$ (resp.
$\bar{\mathbb{A}}=\mathbb{R}/\mathbb{Z}\times [0,1]$) be the open
annulus (resp. the closed annulus). In 1992, Franks \cite{F2} (and
\cite{F3}) proved the following celebrated theorem:
\begin{thm}[Franks]\label{thm:Franks 3.5} Suppose $F$ be an area preserving homeomorphism of the
open or closed annulus which is isotopic to the identity. If $F$ has
at least one fixed or periodic point then $F$ must have infinitely
many interior periodic points.
\end{thm}

Kang \cite[Section 1.2]{Kang} raised the following questions when he
studied the reversible maps on planar domains:

Suppose that $F$ is an area preserving homeomorphism of the open or
closed annulus which is isotopic to the identity. If
$\mathrm{Per}_{\mathrm{odd}}(F)\neq \emptyset$, does it imply that
$\sharp\mathrm{Per}_{\mathrm{odd}}(F)=+\infty$ where
$\mathrm{Per}_{\mathrm{odd}}(F)$ is the set of odd periodic points
of $F$? Furthermore, let $k\in \mathbb{N}$ and $n\in \mathbb{N}$ be
two numbers such that $(n,k)=1$. If
$\mathrm{Per}_{k}(F)\neq\emptyset$, does it imply that
$\sharp\{\bigcup_{(k',n)=1}\mathrm{Per}_{k'}(F)\}=+\infty$ where
$\mathrm{Per}_{k}(F)$ is the set of $k$-prime-periodic points of
$F$? Here $z$ is a $k$-prime-periodic of $F$ means that $z$ is not a
$l$-periodic point of $F$ if $l<k$.\smallskip

In this paper, we answer his questions. We have the following
theorems

\begin{thm}\label{thm:odd}If $\mathrm{Per}_{\mathrm{odd}}(F)\neq \emptyset$, then
$\sharp\mathrm{Per}_{\mathrm{odd}}(F)=+\infty$.
\end{thm}

\begin{thm}\label{thm:k}Assume that $k,n_0\in \mathbf{N}$ which satisfy that $(k,n_0)=1$.
 If $\mathrm{Per}_{k}(F)\neq\emptyset$,
then
$$\sharp\left\{\bigcup_{(k',n_0)=1}\mathrm{Per}_{k'}(F)\right\}=+\infty.$$
\end{thm}

\begin{rem} Theorem \ref{thm:k} implies Theorem \ref{thm:odd}. Indeed, if
$\mathrm{Per}_{\mathrm{odd}}(F)\neq \emptyset$, there is
$k\in2\mathbf{Z}+1$ such that $\mathrm{Per}_{k}(F)\neq\emptyset$.
Taken $n_0=2$, then Theorem \ref{thm:odd} follows from Theorem
\ref{thm:k}.
\end{rem}

Hence, we only need to prove Theorem \ref{thm:k}. We will introduce
some mathematical objects and recall some well-known facts in
Section 2. We will prove Theorem \ref{thm:k} in Section 3.\bigskip

\noindent\textbf{Acknowledgements.} We would like to thank Patrice
Le Calvez and Yiming Long for their helpful conversations and
comments.


\section{Preliminaries}

\subsection{Rotation vector}\label{subsec:rotation vector}
Let us introduce the classical notion of rotation vector which was
defined originally in \cite{S}. Let $M$ be a smooth manifold.
Suppose that $F$ is the time-one map of an identity isotopy
$I=(F_t)_{t\in[0,1]}$ on $M$. Let $\mathrm{Rec}^+(F)$ be the set of
positively recurrent points of $F$. If $z\in \mathrm{Rec}^+(F)$, we
fix an open disk $U\subset M$ containing $z$, and write
$\{F^{n_k}(z)\}_{k\geq 1}$ for the subsequence of the positive orbit
of $z$ obtained by keeping the points that are in $U$. For any
$k\geq 0$, choose a simple path $\gamma_{F^{n_k}(z),z}$ in $U$
joining $F^{n_k}(z)$ to $z$. The homology class $[\Gamma_k]_M\in
H_1(M,\mathbb{Z})$ of the loop $\Gamma_k=
I^{n_k}(z)\gamma_{F^{n_k}(z),z}$ does not depend on the choice of
$\gamma_{F^{n_k}(z),z}$. Say that $z$ has a \emph{rotation vector}
$\rho_{M,I}(z)\in H_1(M,\mathbb{R})$ if
\[\lim_{l\rightarrow
+\infty}\frac{1}{n_{k_l}}[\Gamma_{k_l}]_M=\rho_{M,I}(z)\] for any
subsequence $\{F^{n_{k_l}}(z)\}_{l\geq 1}$ which converges to $z$.
Neither the existence nor the value of the rotation vector depends
on the choice of $U$. Let $\mathcal {M}(F)$ be the set of Borel
finite measures on $M$ whose elements are invariant by $F$. If
$\mu\in\mathcal {M}(F)$ and $M$ is compact, we can define the
rotation vector $\rho_{M,I}(z)$ for $\mu$-almost every positively
recurrent point \cite{P1} (see also Section 1.3 in \cite{W2}). If we
suppose that the rotation vector $\rho_{M,I}(z)$ is
$\mu$-integrable, we define the \emph{rotation vector of the
measure}
$$\rho_{M,I}(\mu)=\int_M\rho_{M,I}\, \mathrm{d}\mu\in H_1(M,\mathbb{R}).$$

Remark that the definition of rotation vector here is the same as
the homological rotation vector that was defined by Franks in
\cite{F2} on the positively recurrent points set. \bigskip

The following theorem is due to Franks (\cite{F2,F3}):

\begin{thm}\label{thm:Franks 3.4}Let $M$ be an oriented surface of
genus $0$ with $\chi(M)\leq0$, that is,
$M=\mathbf{S}^2\setminus\{x_1,x_2,\cdots,x_n\}$ where $n\geq2$.
Suppose that $F$ is the time-one map of an identity isotopy
$I=(F_t)_{t\in[0,1]}$ on $M$ and preserves a finite measure $\mu$ of
$M$ with total support. If $\rho_{M,I}(\mu)=0$, then $F$ has a fixed
point in the interior of $M$.
\end{thm}

Denote by $\mathrm{Fix}_{\mathrm{Cont},I}(F)$ the set of
contractible fixed points of $F$, that is,
$x\in\mathrm{Fix}_{\mathrm{Cont},I}(F)$ if and only if $x$ is a
fixed point of $F$ and the oriented loop $I(x): t\mapsto F_t(x)$
defined on $[0,1]$ is contractible on $M$.  In \cite[Theorem
8.1]{P1}, Le Calvez proved the following deep result
\begin{thm}\label{thm:foliation}
Suppose that $M$ is a surface (without boundary) and
$I=(F_t)_{t\in[0,1]}$ is an isotopy on $M$ from $\mathrm{Id}_{M}$ to
$F$. We suppose that $F$ has no contractible fixed points. Then
there exists an oriented topological foliation $\mathcal{F}$ on $M$
such that, for all $z\in M$, the trajectory $I(z)$ is homotopic to
an arc $\gamma$ joining $z$ and $F(z)$ in $M$ which is positively
transverse to $\mathcal{F}$. That means that for every $t_0\in[0,1]$
there exists an open neighborhood $V\subset M$ of $\gamma(t_0)$ and
an orientation preserving homeomorphism $h:V\rightarrow(-1,1)^2$
which sends the foliation $\mathcal{F}$ on the horizontal foliation
(oriented with $x_1$ increasing) such that the map $t\mapsto
p_2(h(\gamma(t)))$ defined in a neighborhood of $t_0$ is strictly
increasing where $p_2(x_1,x_2)=x_2$.
\end{thm}

We say that $X\subseteq \mathrm{Fix}_{\mathrm{Cont},I}(F)$ is
\emph{unlinked}
 if there exists an isotopy $I'=(F'_t)_{t\in [0,1]}$ homotopic
to $I$ which fixes every point of $X$, that is, $F'_t(x)=x$ for all
$t\in[0,1]$ and $x\in X$. Moreover, we say that $X$ is a
\emph{maximal unlinked set}, if any set $X'\subseteq
\mathrm{Fix}_{\mathrm{Cont},I}(F)$ which strictly contains $X$ is
not unlinked. If $\sharp\mathrm{Fix}_{\mathrm{Cont},I}(F)<\infty$,
there must be a set $X\subseteq \mathrm{Fix}_{\mathrm{Cont},I}(F)$
which is a maximal unlinked set. By Theorem \ref{thm:foliation},
there exists an oriented topological foliation $\mathcal{F}$ on
$M\setminus X$ (or, equivalently, a singular oriented foliation
$\mathcal{F}$ on $M$ with $X$ equal to the singular set) such that,
for all $z\in M\setminus X$, the trajectory $I(z)$ is homotopic to
an arc $\gamma$ joining $z$ and $F(z)$ in $M\setminus X$ which is
positively transverse to $\mathcal{F}$.

\subsection{Rotation number}We denote by $\pi$ the covering map of the open annulus
(resp. the closed annulus)
\begin{eqnarray*}
\pi\,:\, \mathbb{R}\times(0,1)\quad
(\mathrm{resp.}\quad\mathbb{R}\times[0,1])&\rightarrow&
\mathbb{A}\quad
(\mathrm{resp.}\quad\bar{\mathbb{A}})\\
(x,y)&\mapsto&(x+\mathbb{Z},y),
\end{eqnarray*}
and by $T$ the generator of the covering transformation group
\begin{eqnarray*}
T\,:\, \mathbb{R}\times(0,1)\quad (\mathrm{resp.}\quad\mathbb{R}\times[0,1])&\rightarrow& \mathbb{R}\times(0,1)\quad (\mathrm{resp.}\quad\mathbb{R}\times[0,1]) \\
(x,y)&\mapsto&(x+1,y).
\end{eqnarray*}

Write respectively $S$ and $N$ for the lower and the upper end of
$\mathbb{A}$.\smallskip

In the following, we denote by $\mathbb{A}$ the open or closed
annulus unless an explicit mention. We call \emph{essential circle}
in $\mathbb{A}$ every simple closed curve which is not
null-homotopic. Let $F$ be a homeomorphism of $\mathbb{A}$. We say
that $F$ satisfies the \emph{intersection property} if  any
essential circle in $\mathbb{A}$ meets its image by $F$. We denote
the space of all homeomorphisms of $\mathbb{A}$ which are isotopic
to the identity as $\mathrm{Homeo}_*(\mathbb{A})$ and its subspace
whose elements additionally have the intersection property as
$\mathrm{Homeo}_*^\wedge(\mathbb{A})$. It is easy to see that a
homeomorphism of $\mathbb{A}$ that preserves a finite measure with
total support satisfies the intersection property.

When $F\in\mathrm{Homeo}_*(\mathbb{A})$, we define the rotation
number of a positively recurrent point as follows. We say that a
positively recurrent point $z$ has a \emph{rotation number}
$\rho(f;z)\in \mathbb{R}$ for a lift $f$ of $F$ to the universal
cover of $\mathbb{A}$, if for every subsequence
$\{F^{n_k}(z)\}_{k\geq 0}$ of $\{F^n(z)\}_{n\geq 0}$ which converges
to $z$, we have
\[\lim_{k\rightarrow+\infty}\frac{p_1\circ
f^{n_k}(\widetilde{z})-p_1(\widetilde{z})}{n_k}=\rho(f;z)\] where
$\widetilde{z}\in \pi^{-1}(z)$ and $p_1$ is the first projection
$p_1(x,y)=x$. In particular, the rotation number $\rho(f;z)$ always
exists and is rational when $z$ is a fixed or periodic point of $F$.
Let $\mathrm{Rec}^+(F)$ be the set of positively recurrent points of
$F$. We denote the set of rotation numbers of positively recurrent
points of $F$ as $\mathrm{Rot}(f)$.

It is well known that a positively recurrent point of $F$ is also a
positively recurrent point of $F^q$ for all $q\in \mathbb{N}$ (see
the appendix of \cite{W}). By the definition of rotation number, we
easily get that the following elementary properties.
\begin{enumerate}\label{prop:ROT}
  \item[1.] $\rho(T^k\circ f;z)=\rho(f;z)+k$, and hence $\mathrm{Rot}(T^k\circ f)=\mathrm{Rot}(f)+k$
  for every $k\in \mathbb{Z}$;
  \item[2.] $\rho(f^q;z)=q\rho(f;z)$, and hence $\mathrm{Rot}(f^q)=q\mathrm{Rot}(f)$ for every $q\in \mathbb{N}$.
\end{enumerate}\smallskip

Suppose that $z\in \mathrm{Rec}^+(F)$ and $\widetilde{z}\in
\pi^{-1}(z)$. We define $\mathcal
{E}(z)\subset\mathbb{R}\cup\{-\infty,+\infty\}$ by saying that
$\rho\in\mathcal {E}(z)$ if there exists a sequence
$\{n_k\}_{k=1}^{+\infty}\subset\mathbb{N}$ such that
\begin{eqnarray*}
   &\bullet& \lim_{k\rightarrow+\infty}F^{n_k}(z)=z;\qquad\qquad\qquad\qquad\qquad\qquad\qquad\qquad
   \qquad\qquad\qquad\qquad\qquad\qquad\quad \\
   &\bullet& \lim_{k\rightarrow+\infty}\frac{p_1(f^{n_k}(\widetilde{z}))-p_1(\widetilde{z})}{n_k}=\rho.
\end{eqnarray*}

Define $\rho^-(f;z)=\inf\mathcal
{E}(z)\quad\mathrm{and}\quad\rho^+(f;z)=\sup\mathcal {E}(z).$
Obviously, we have that $\rho(f;z)$ exists if and only if
$\rho^-(f;z)=\rho^+(f;z)\in\mathbb{R}$. Note that the set $\mathcal
{E}(z)$ is a bounded set when $\mathbb{A}$ is a closed annulus (by
compactness) and might be an unbounded set when $\mathbb{A}$ is an
open annulus. However, the set $\mathcal {E}(z)$ is still a bounded
set if $\sharp\mathrm{Fix(F)}<+\infty$ (see \cite{P2}). By the
definitions, it is easy to see that $\rho^-(f;z)$ and $\rho^+(f;z)$
satisfy the same properties as $\rho(f;z)$.

\begin{rem}\label{rem:rot number} A lift $f$ of $F$ is one-to-one
corresponding to an identity isotopy $I$ (mod homptopy). Observing
that $H_1(\mathbb{A},\mathbb{R})\simeq\mathbb{R}$, the rotation
number $\rho(f;z)$ is nothing else but the rotation vector
$\rho_{\mathbb{A},I}(z)$ where the time-one map of the lift identity
isotopy of $I$ to the universal cover is $f$.
\end{rem}

The following Theorem is due to Franks \cite{F1} when $\mathbb{A}$
is closed annulus and $F$ has no wandering point,  and it was
improved by Le Calevez \cite{P1} (see also \cite{W}) when
$\mathbb{A}$ is open annulus and $F$ satisfies the intersection
property:

\begin{thm}\label{thm:FP}Let $F\in \mathrm{Homeo}_*^\wedge(\mathbb{A})$
 and $f$ be a lift of $F$ to the universal cover of $\mathbb{A}$. Suppose that there
exist two recurrent points $z_{1}$ and $z_{2}$ such that
$-\infty\leq\rho^-(f;z_{1})<\rho^+(f;z_{2})\leq+\infty$. Then for
any rational number $p/q\in ]\rho^-(f;z_{1}),\rho^+(f;z_{2})[$
written in an irreducible way, there exists a periodic point of
period $q$ whose rotation number is $p/q$.
\end{thm}

\bigskip

\section{Proof of the theorems}
\begin{proof}[Proof of the Theorem \ref{thm:k}]
If $\sharp\mathrm{Per}_k(F)=+\infty$, we have nothing to do. Hence
we assume that $\sharp\mathrm{Per}_k(F)<+\infty$. Let
$\mathrm{Per}_k(F)=\{x_1,\cdots,x_m\}$ and $G=F^k$. If
$\sharp\mathrm{Fix}(G)=+\infty$, then there exists $t\,|\,k$ with
$(t,n_0)=1$ such that $\sharp\mathrm{Per}_t(F)=+\infty$. We have
done. Therefore, we can assume that
$1\leq\sharp\mathrm{Fix}(G)<+\infty$. Assume that
$\mathrm{Fix}(G)=\{y_1,\cdots,y_{m'}\}$. By Theorem \ref{thm:FP}, we
can choose a lift $f$ of $F$ to the universal cover of $\mathbb{A}$
or $\bar{\mathbb{A}}$ such that $\rho(f;x_i)=\frac{k'}{k}\in[0,1)$
for $i=1,\cdots,m$ where $k'\in\mathbb{Z}$. The proof will be
divided into two cases: $k'=0$ and $k'\neq0$. \smallskip

In the case of $k'\neq0$, by Theorem \ref{thm:Franks 3.5}, we know
that $\sharp\mathrm{Per}(G)=+\infty$. Hence
$\mathrm{Per}(G)\setminus\mathrm{Fix}(G)\neq\emptyset$. Note that
$\rho(f^k;y_i)=k'$ for all $i$. For any
$z\in\mathrm{Per}(G)\setminus\mathrm{Fix}(G)$, we assume that
$\rho(f^k;z)=\frac{p}{q}$ with $q\geq1$ and $(p,q)=1$ if $p\neq0$.
W.l.o.g., we assume that $k'<\frac{p}{q}$. Then by Theorem
\ref{thm:FP}, for every $\frac{r}{s}\in(k',\frac{p}{q})$ with
$(r,s)=1$ and $(s,n_0)=1$, there exists $z'\in\mathrm{Per}(G)$ such
that $\rho(f^k;z')=\frac{r}{s}$. Observing that the point $z'$ is a
$ks$-periodic point of $F$, the conclusion follows in this
case.\bigskip

We now consider the case of $k'=0$.  If the annulus is a closed
annulus $\bar{\mathbb{A}}$, we consider its interior $\mathbb{A}$.
Let $\mathbf{S}^2=\mathbb{A}\sqcup\{N,S\}$. Note that
$\chi(\mathbf{S}^2\setminus\{N,S,y_1,\cdots,y_{m'}\})\leq0$ (=0 when
the annulus is closed and $\mathrm{Fix}(G)\subset \partial
\bar{\mathbb{A}}$). Write
$M=\mathbf{S}^2\setminus\{N,S,y_1,\cdots,y_{m'}\}$.\smallskip

When $\chi(M)=0$, we work on the closed annulus $\bar{\mathbb{A}}$.
We have that $\rho(f^k,y_i)=0$ and $\rho(f^k,z)=0$ for any
$z\in\mathrm{Rec^+(G)}$ (by Theorem \ref{thm:FP}). By Theorem
\ref{thm:Franks 3.4}, there is a fixed point of $G$ in the interior
of $\bar{\mathbb{A}}$ which contradicts the fact that
$\mathrm{Fix}(G)\subset
\partial \bar{\mathbb{A}}$.\smallskip

In the case of $\chi(M)<0$, we follow the idea of Le Calvez
\cite[Theorem 9.3]{P1}.
 We choose an identity isotopy
 $I_0=(F_t)_{t\in[0,1]}$ on $\mathbf{S}^2$ such that $F_t$ fixes $N$ and $S$ for every $t$, and the
time-one map of the lift identity isotopy of $I_0|_{\mathbb{A}}$ to
$\mathbb{R}\times(0,1)$ is $f^k$. Furthermore, as $\rho(f^k,y_i)=0$
for all $i$, we can suppose that $I_0$ fixes $N,S$ and one point of
$\{y_1,\cdots,y_{m'}\}$ (e.g., we can modify $I_0$ though the
technique in the proof of Lemma 1.2 in \cite[Section 1.4]{W2}
without changing the homotopic class of $I_0|_{\mathbb{A}}$).
Identify $G$ as a map of $\mathbf{S}^2$. As
$\sharp\mathrm{Fix}(G)<+\infty$, there is a maximal unlink set
$X\subset \mathrm{Fix}(G)=\{N,S,y_1,\cdots,y_{m'}\}$ and an identity
isotopy $I_1$ which is homotopic to $I_0$ with fixed endpoints such
that $I_1$ fixes every point of $X$. Note that $\sharp X\geq3$ in
this case. By Theorem \ref{thm:foliation}, there exists an oriented
topological foliation $\mathcal{F}$ on $\mathbf{S}^2\setminus X$
such that, for all $z\in \mathbf{S}^2\setminus X$, the trajectory
$I_1(z)$ is homotopic to an arc $\gamma$ joining $z$ and $G(z)$ in
$\mathbf{S}^2\setminus X$ which is positively transverse to
$\mathcal{F}$. For any leaf $\lambda\in \mathcal{F}$, the
$\alpha$-limit set and $\omega$-limit set of $\lambda$ must belong
to two distinct points of $X$ respectively since $X$ is finite and
$G$ is symplectic. Choose a leaf $\lambda\in \mathcal{F}$ which
connects two different points $z_1$ and $z_2$ of $X$. We consider
the following open annulus
$A_{z_1,z_2}=\mathbf{S}^2\setminus\{z_1,z_2\}$.

We choose a small open disk $U$ near $\lambda$ such that, $U\cap
\lambda=\emptyset$ and for any $z\in U$, $\lambda\wedge I_1(z)\geq
1$ where $I_1(z)$ is the trajectory of $z$ under the isotopy $I_1$.
 We define the first
return map

\begin{eqnarray*}
  \Phi: \mathrm{Rec}^+(G)\cap U&\rightarrow& \mathrm{Rec}^+(G)\cap U, \\
  z &\mapsto& G^{\tau(z)}(z),
\end{eqnarray*}
where $\tau(z)$ is the first return time, that is, the least number
$n\geq1$ such that $G^n(z)\in U$. By Poincar\'{e} Recurrence
Theorem, this map is defined $\mu$-a.e. on $U$. For every couple
$(z',z'')\in U^2$, choose a simple path $\gamma_{z',z''}$ in $U$
joining $z'$ to $z''$. For every $z\in \mathrm{Rec}^+(G)\cap U$ and
$n\geq 1$, define
$$\tau_n(z)=\sum_{i=0}^{n-1}\tau(\Phi^i(z)),\quad \Gamma_z^n=I_1^{\tau_n(z)}(z)\gamma_{\Phi^n(z),z},
\quad m(z)=\Gamma_z^1\wedge\lambda,\quad
m_n(z)=\sum_{i=0}^{n-1}m(\Phi^i(z)).$$

It is well known that $\tau\in L^1(U,\mathbb{R})$ (see, e.g.,
\cite[Section 1.3]{W2}). Hence $\tau_n/n$ converges $\mu$-a.e. on
$\mathrm{Rec}^+(G)\cap U$. It is clear that $m_n/n\geq1$ for all
$n\geq1$ and $z\in \mathrm{Rec}^+(G)\cap U$. This implies that
$m_n/\tau_n>0$ for $\mu$-a.e. on $\mathrm{Rec}^+(G)\cap U$. Observe
that
$$\mathcal {E}(z)\subset[\inf_n\left\{\frac{m_n(z)}{\tau_n(z)}\right\},
\sup_n\left\{\frac{m_n(z)}{\tau_n(z)}\right\}]\subset \mathbb{R}$$
when the limit of $\tau_n(z)/n$ exists, where the definition
$\mathcal {E}(z)$ one can refer to Remark \ref{rem:rot number}. We
get that $\rho^-_{A_{z_1,z_2},I_1}(z)>0$ for $\mu$-a.e. on
$\mathrm{Rec}^+(G)\cap U$. We also have that
$\rho_{A_{z_1,z_2},I_1}(y)=0$ for all $y\in X\setminus\{z_1,z_2\}$.
Then Theorem \ref{thm:k} follows by Theorem \ref{thm:FP}.
\end{proof}

\end{document}